\documentclass{article}
\usepackage{amsmath,latexsym,amsbsy,amssymb,graphicx}
\usepackage{subfigure}
\usepackage{amsthm}

\def\RR{{\mathbb R}}

\newcommand{\h}{\hspace*{.24in}}
\newcommand{\hhh}{\hspace*{.72in}}
\newtheorem{thm}{Theorem}[section]
\newtheorem{lem}{Lemma}[section]
\newtheorem{rem}{Remark}[section]
\newtheorem{cor}{Corollary}[section]

\title{Convergence of Griddy Gibbs Sampling  and other perturbed Markov chains}

\author{
Vu Dinh\thanks{Fred Hutchinson Cancer Research Center, Seattle, Washington, USA}
\and Ann E. Rundell\thanks{Weldon School of Biomedical Engineering, Purdue University, West Lafayette, Indiana, USA}
\and Gregery T. Buzzard\thanks{Department of Mathematics, Purdue University, West Lafayette, Indiana, USA}
}

\date{August 16, 2016}

\begin{document}

\maketitle

\begin{abstract}
The Griddy Gibbs sampling was proposed by Ritter and Tanner (1992) as a computationally efficient approximation of the well-known Gibbs sampling method. The algorithm is simple and effective and has been used successfully to address problems in various fields of applied science. However, the approximate nature of the algorithm has prevented it from being widely used: the Markov chains generated by the Griddy Gibbs sampling method are not reversible in general, so the existence and uniqueness of its invariant measure is not guaranteed. Even when such an invariant measure uniquely exists, there was no estimate of the distance between it and the probability distribution of interest, hence no means to ensure the validity of the algorithm as a means to sample from the true distribution.

In this paper, we show, subject to some fairly natural conditions, that the Griddy Gibbs method has a unique, invariant measure.  Moreover, we provide $L^p$ estimates on the distance between this invariant measure and the corresponding measure obtained from Gibbs sampling.  These results provide a theoretical foundation for the use of the Griddy Gibbs sampling method. We also address a more general result about the sensitivity of invariant measures under small perturbations on the transition probability. That is, if we replace the transition probability $P$ of any Monte Carlo Markov chain by another transition probability $Q$ where $Q$ is close to $P$, we can still estimate the distance between the two invariant measures. The distinguishing feature between our approach and previous work on convergence of perturbed Markov chain is that by considering the invariant measures as fixed points of linear operators on function spaces, we don't need to impose any further conditions on the rate of convergence of the Markov chain. For example, the results we derived in this paper can address the case when the considered Monte Carlo Markov chains are not uniformly ergodic.
\end{abstract}

\maketitle

\newpage

\section{Introduction}

The need to generate samples from a probability function or estimate moments of such a distribution arises in many fields of applied science, including Bayesian statistics, computational physics, computational biology and computer science. A common difficulty in generating such samples is that the distribution (hereafter denoted by $\pi$) may be high-dimensional and computationally intractable.  To resolve this problem, many sampling-based approaches have been proposed: the basic idea is to construct a Markov chain with a tractable transition mechanism that has $\pi$ as its invariant distribution.

One of the most widely applicable methods to construct such a Markov chain is the method of Gibbs sampling. This algorithm generates an instance from the distribution of each variable in turn, conditional on the current values of the other variables. This reduces the sampling problem to a series of one-dimensional problems. The method of Gibbs sampling is very computationally effective, especially in the case when $\pi$ is high-dimensional. Gibbs sampling applies even in the case that the distribution is known only up to a normalizing constant, which occurs commonly in fitting models to data.  

However, the use of the Gibbs sampling method is hindered by several factors. 
First, the method requires the one-dimensional conditional densities to be known, or at least to be easy to sample directly. 
In most contexts, such knowledge about the conditional densities is usually not available. 
Second, in many fields of applied sciences, sampling from the conditional distributions is computationally expensive, despite the fact that they are one-dimensional. 
For instance, in systems biology, evaluating (up to a normalizing factor) the value of the distribution function $\pi$ at one point might be equivalent to solving a high-dimensional system of differential equations. 
These high computational costs pose a serious problem in applying the algorithm in practice, which motivates the development of sampling methods which uses approximation of the full distribution to alleviate this difficulty.

To address these issues, Ritter and Tanner (1992) proposed in \cite{RT} an approximate method -- the Griddy Gibbs method -- as an alternative. The Griddy Gibbs sampling method evaluates the conditional density on a grid and uses piecewise linear or piecewise constant functions to approximate the cumulative distribution function of the conditional distributions based on these grid values.  The resulting distribution is used to generate random variables with approximately the right distribution.

This method has been used successfully to address problems in various fields of applied science: statistical modeling and inference \cite{A9,A10, YH, KD1, KD2, POW}, machine learning \cite{MV}, chemical analysis \cite{A14}, systems biology \cite{D,DRB1, DRB2}, medical science \cite{A1}, statistical computing and data analysis \cite{A6,A8}, economics \cite{A2,A11,A13}, ecological modelling \cite{A12}, acoustics \cite{A3}, and time series analysis \cite{A4,A5,A7}. However, the approximate nature of the algorithm still prevents it from being widely used. The  approximation by linear or constant functions leads to theoretical questions about the ergodic properties of the constructed Markov chains and about the validity of the algorithm as a means to sample from the true distribution. 

Many adjustments to overcome the approximate nature of the algorithm have been proposed. In \cite{T}, a Metropolis chain is embedded in the algorithm to ensure that the equilibrium distribution is exactly $\pi$ even on a coarse grid. In \cite{L}, a similar strategy is proposed, in which the Multiple-try Metropolis algorithm is embedded in the sampling process. In both approaches, the convergence of the algorithms are guaranteed, but the computational costs increase considerably and the algorithms are more difficult to set up.
Another popular grid-based sampling algorithm is the adaptive rejection sampler (ARS), which is capable of sampling the distribution efficiently with a low number of grid points \cite{GW}. 
However, ARS-type algorithms are designed for log-concave (or near log-concave) distributions and in general does not work well for general distribution with multiple modes.
In all cases, the approximations are restricted to piecewise linear and piecewise constant functions. 

In this paper, we show, assuming that the approximations to the distribution are bounded from above and bounded away from zero, that the Griddy Gibbs method has a unique, invariant measure. Moreover, we provide $L^p$ estimates on the distance between this invariant measure and the corresponding measure obtained from Gibbs sampling.  Subject to appropriate hypotheses, our main results about Griddy Gibbs are the following.  

\begin{enumerate}
\item Although the Markov chains generated by the Griddy Gibbs sampler are not reversible in general, they admit unique invariant measures.
\item For $2 \le p \le \infty$, there is an $L^p$-estimate of the distance between the limit invariant measure and the correct distribution $\pi$, which guarantees the $L^p$-convergence of the algorithm.
\end{enumerate}

The first result is obtained using tools from the theory of Markov processes.  We then extend the Markov chain transition operator to $L^p$-spaces and use techniques from the theory of functional analysis on Hilbert vector spaces to prove the second result.  These results provide a theoretical foundation for the use of the Griddy Gibbs sampling method with some guarantees of convergence. 

In this paper, we go beyond these results for Griddy Gibbs sampling in two ways. First, the approximation scheme does not need to be piecewise linear or piecewise constant: any reasonable approximation scheme can be employed to obtain Griddy Gibbs sampling. In fact, in the case that the distribution is smooth but the computational cost of determining the value of the conditional distribution is much greater than the cost for approximation, high order polynomial interpolations are preferred since they increase the accuracy of the sampling process and reduce the number of function evaluations. 

Second, we generalize our method to give results about the sensitivity of invariant measures under small perturbations on the transition probability. That is, if we replace the transition probability $P$ of any Monte Carlo Markov chain by another transition probability $Q$ where $Q$ is close to $P$, can we still estimate the distance between the two invariant measures? 
Our paper provides a positive answer to this question, given some mild conditions imposed on Q. 
The distinguishing feature between our approach and other work \cite{R1,R2,R3} on convergence of perturbed Markov chain is that by considering the invariant measures as fixed points of linear operators on function spaces, we don't need to impose any further conditions on the rate of convergence of the Markov chain. 
For example, the results we derive in this paper can address the case when the considered Monte Carlo Markov chains are not uniformly ergodic.
The result about the sensitivity of invariant measures under small perturbations thus guarantees the convergence and provides a way to assess the efficiency of  Griddy Gibbs sampler and other perturbed Markov chains methods. 

The paper is organized as follows. Section 2 provides the mathematical framework used in the paper, as well as descriptions of the Gibbs and Griddy Gibbs sampling methods. Section 3 discusses the existence, uniqueness and regularity results for the invariant measure. We develop in Section 4 results about the sensitivity of invariant measures under small perturbations on the transition probability for general non-uniformly ergodic Monte Carlo Markov chains. The estimates are then extended to the case when the distribution of interest has non-compact support in Section 5. Finally, we provide in Section 6 numerical examples to illustrate our theoretical findings and demonstrate the utility of the Griddy Gibbs sampling method.

\section{Mathematical framework}

The problem addressed by the Gibbs algorithm is the following (see \cite{RR} for reference). We are given a density function $\hat{\pi}$, on a state space with bounded Lebesgue measure $D \subset \RR^d$. This density gives rise to an absolutely continuous probability measure $\pi$ on D, by
\[
\pi(A)=\int_{A}{\hat{\pi}(x)dx}, ~~ \forall A \in \mathcal{B}
\]
where $\mathcal{B}$ denotes the $\sigma$-algebra of Borel sets on D. Without loss of generality, we assume throughout this paper that the the distribution $\pi$ has finite variance. In other words, the density function $\hat \pi \in L^2(D)$.

In many applications, we want to estimate the expectations of functions $\phi:D \to \RR$ with respect to $\pi$, i.e. we want to estimate
\[
\pi(\phi)=E_{\pi}[\phi(X)]=\int_{D}{\phi(x) \hat{\pi}(x)dx}.
\]
If $D$ is high-dimensional, and $\hat \pi$ is a complicated function, then direct integration (either analytic or numerical) of these integrals is infeasible. 

The classical Monte Carlo solution to this problem is to simulate independent and identically distributed random variables $X_1,X_2,...,X_N$ with distribution $\pi$ and then estimate $\pi(\phi)$ by
\begin{equation} 
\label{sumPhi}
\pi_N(\phi)= \frac{1}{N} \sum_{i=1}^{N}{\phi(X_i)}.
\end{equation}
This gives an unbiased estimate with standard deviation of order $O(1/\sqrt{N})$. However, if $\pi_u$ is complicated, it is difficult to directly simulate i.i.d. random variables from $\pi$. 
The Markov chain Monte Carlo (MCMC) approach is introduced instead to construct on $D$ a Markov chain that is computationally efficient and that has $\pi$ as a stationary distribution. 
That is, we want to define easily-simulated Markov chain transition probabilities $P(x,y)$ for $x,y \in D$ such that
\[
\int_{D}{\hat{\pi}(y)P(x,y)dy}=\hat{\pi}(x), ~~ for ~ a.e. ~x \in D.
\]
In principle, if we run the Markov chain (started from anywhere) to obtain samples $X_n$, then for large $n$ the distribution of $X_n$ will be approximately stationary, and the sequence $\{X_n\}$ can be used to estimate $\pi(\phi)$ as in equation \eqref{sumPhi}. 

\subsection{Gibbs transition}
The Gibbs transition is a transition probability on $D$ defined as follows. 
First, the $i^{th}$ component Gibbs transition $P_i$ leaves all components except the $i^{th}$ component unchanged and replaces the $i^{th}$ component by a draw from the full distribution $\pi$ conditional on all other components:
\[
	P_i(x_1,\ldots,x_i,\ldots,x_d)= \frac{\hat \pi(x_1,\ldots,x_i,\ldots,x_d)}{\int_{-\infty}^{\infty}{\hat \pi(x_1,\ldots,t,\ldots,x_d) dt}},
\]
where $t$ appears in the $i^{th}$ position.
The transition probability of the Gibbs sampler is defined as
\[
P(x,y)= P_1(y_1,x_2,\ldots,x_d)P_2(y_1,y_2,x_3,\ldots,x_d) \cdots P_d(y_1,y_2,\ldots,y_d)
\]
where $x=(x_1,x_2,\ldots,x_d)$ and $y=(y_1,y_2,\ldots,y_d)$.

Now let $\{X_n\}, n \ge 0$ be a time-homogeneous Markov process generated by the Gibbs sampling algorithm with transition probability $P$.  
We have
\[
P(X_n \in A| X_0 = a) = P^n(a,A), ~~ \forall A \in \mathcal{B}
\]
where $P^n$ is defined recursively by
\[
P^1 = P , \h  P^n(a,y) = \int_{D}{P(x,y) P^{n-1}(a, x)dx}.
\]
We also define the transition operator $T$ on $\mathcal{P}$(D), the space of probability measures on D, by
\begin{equation}
T\mu(A)=\int_D{P(x,A)\mu(dx)}.
\label{E1}
\end{equation}
This transition operator can also be considered as a linear operator on $L^p(D)$, $1 \le p \le \infty$, by defining
\[
Tf(y)=\int_D{P(x,y)f(x)dx}
\]
Moreover, the operator $T^n$ obtained by replacing $P$ by $P^n$ in ($\ref{E1}$) is equal to the operator obtained by applying $T$ $n$ times, $T^n = T \circ T \circ ...\circ T$.
	
\subsection{Ergodic properties of the Markov chains generated by the Gibbs sampling}
By standard results about ergodicity of Gibbs sampling method, we know that under rather general conditions, $T$ admits a unique invariant measure, which is the distribution $\pi$ that we want to sample, i.e. $T \pi=\pi$. Moreover, the distribution of ${X_n}$ converges in total variation norm to $\pi$. We state here Theorem 6 from \cite{A} that justifies the convergence of the Gibbs sampling method:
\begin{thm}[\cite{A}]
Assume that for each $1 \le i \le d$, the conditional distributions $\pi(X_i | X_j,j\neq i)$ have densities, say $p_i$, with respect to some dominating measure $\rho_i$. Suppose further that for each $1 \le i \le d$, there is a set $A_i$ with $\rho_i(A_i)>0$, and a $\delta>0$ such that for each $1 \le i \le d$ 
\begin{enumerate}
	\item $\pi(X_i=x_i | X_j=x_j,j\neq i)>0$
	whenever
	$x_k \in A_k$ for all $k \leq i$ and $ x_{i+1},...,x_d$  arbitrary.
	\item $\pi(X_i=x_i | X_j=x_j,j\neq i)>\delta$
	whenever
	$x_k \in A_k$ for all $k \leq d$.
\end{enumerate}
Then for $\pi$-a.e. $x \in D$, we have $\sup_{C \in \mathcal{B}}{|P^n(x,C)- \pi(C)| } \to 0$. 
\label{TheoGibb}
\end{thm}
In the rest of the paper, we will assume that the distribution of interest $\pi$ satisfies conditions of Theorem $\ref{TheoGibb}$ and has finite variance.

\subsection{Griddy Gibbs transition and Griddy Gibbs sampler}
In the Griddy Gibbs sampling method, at each point in the sampling space and on each dimension, we use some approximation scheme to approximate $P_i$. 
The $i^{th}$ component Griddy Gibbs transition leaves all components except the $i^{th}$ component unchanged and replaces the $i^{th}$ component by a draw from $Q_i$ that approximates the conditional expectation on all other components, i.e., 
\begin{equation}
Q_i(x_1,...,y_i,...,x_d)\approx \frac{\hat \pi(x_1,...,y_i,...,x_d)}{\int_{-\infty}^{\infty}{\hat \pi(x_1,...,t,...,x_d) dt}}. 
\label{21}
\end{equation}

The Griddy Gibbs sampler is defined in a similar manner as in the definition of Gibbs sampler: starting with an initial value, we sequentially and randomly update the $i^{th}$ component while fixing other components of the variables. 
However, for each $i^{th}$ component, instead of the conditional distribution $P_i$ (which is unknown or difficult to compute), we use the approximation $Q_i$ as the guideline for the sampling process. 
The surrogate function $Q_i$ is obtained by evaluating the conditional density $P_i$ on a grid and uses some interpolation to approximate $P_i$ based on these grid values.
A simple one-dimensional simple, such as the inverse transform sampler, is then used to sample the $i^{th}$ component corresponding to $Q_i$. 

As with Gibbs sampler, the transition probability and transition operator of the Griddy Gibbs are defined as
\begin{equation}
Q(x,y)= Q_1(y_1,x_2,...,x_d)Q_2(y_1,y_2,...,x_d)...Q_d(y_1,y_2,...,y_d)
\label{22}
\end{equation}
and 
\[
Q^n(a,y) = \int_{D}{Q(x,y) Q^{n-1}(a, x)dx} \h \text{and} \h S\mu(A)=\int_{D}{Q(x,A)\mu(dx)}.
\]
We note that since the approximations on each dimension are different, the Markov chain $\{Y_n\}$ generated by Griddy Gibbs algorithm is not reversible in general.

Throughout this paper, we will use the notation $\{X_n\}$, T, P to describe a Markov chain generated by Gibbs sampling, its transition operator and its transition probability, respectively. The corresponding notations for Griddy Gibbs are $\{Y_n\}$, S, Q. A comparison between the notations used for the Gibbs sampling and Griddy Gibbs sampling is provided in Figure~\ref{fig1}.

\begin{figure}
\[
\text{Gibbs:}\hhh \{X_n\}\stackrel{T,P}{\longrightarrow} \pi
\]
\[
\text{Griddy Gibbs:}\h \{Y_n\}\stackrel{S,Q}{\longrightarrow} \eta
\]
\caption{\footnotesize Comparison between Gibbs sampling and Griddy Gibbs sampling: Although the two transition operators $P$ and $Q$ are close, the Markov chain $\{Y_n\}$ is not reversible in general, so the existence and uniqueness of the invariant measure $\eta$ is not guaranteed. Even when $\eta$ uniquely exists, an estimate of the distance between $\pi$ and $\eta$ is needed to guarantee the validity of the Griddy Gibbs sampling.}
\label{fig1}
\end{figure}

\section{Existence, uniqueness, and regularity of the invariant measure of a Monte Carlo Markov chain generated by the Griddy Gibbs sampling}

In this section, we will prove that the transition operator $S$ (obtained from the Griddy Gibbs algorithm as in ($\ref{21}$), ($\ref{22}$)) admits a unique invariant measure $\eta$, assuming that the approximations $Q_i$ are uniformly bounded above and away from zero:  
 \begin{equation} 
 \exists M, \epsilon >0, \text{ such that }  \epsilon \le Q_i(x) \le M, ~ \forall 1 \le i \le d, ~ \forall x\in D.
 \label{Con1}
 \end{equation}
We also prove that under this condition, $\eta$ is absolutely continuous with respect to  Lebesgue measure and admits a bounded density function. 
We note that condition \eqref{Con1} is general and does not hinder the application of the Griddy Gibbs sampling method, since we can always use additional cutoff functions on the approximation scheme to guarantee the boundedness from above and below of $f_i$, without significantly affecting the accuracy of the approximation scheme.

The outline of the proof is as follow. 
By verifying Doeblin's condition (see Theorem $\ref{C1}$), we prove the existence and uniqueness of the invariant measure $\eta$; moreover, the distribution of $\{Y_n\}$ (obtained by the Griddy Gibbs algorithm) converges to $\eta$ in total variation norm. 
Using this and Lemma $\ref{L1}$, we deduce that $\eta$ is absolutely continuous with respect to Lebesgue measure. Finally, using Lemma $\ref{L2}$, we prove that the density function of $\eta$ is bounded. 

\subsection{Existence and uniqueness}
To verify the existence and uniqueness of the invariant measure, we use the following result from \cite{KS} on the convergence of transition probabilities. As before, we will denote by $\mathcal{B}$ the $\sigma$-algebra of Borel sets on $D$.

\begin{thm}[\cite{KS}]
Suppose that the Markov chain ${Z_n}$ with transition probability $K(x,\cdot)$ satisfies the Doeblin condition:

$\exists k \in N, \epsilon >0$, and a probability measure $\phi$ on $(D,\mathcal{B})$ such that
\[
K^k(x,C) \ge \epsilon \phi(C), \forall x\in D, \forall C \in \mathcal{B}.
\]
Then there exists a unique invariant probability measure $\xi$ such that for all $n \in N$ and all $x \in D$,
\[
\sup_{C \in \mathcal{B}}{|K^n(x,C)- \xi(C)| } \le (1-\epsilon)^{((n/k) -1)}.
\]
\label{C1}
\end{thm}
Using this result, we can prove that under condition $\eqref{Con1}$, the distribution of the Markov chain $\{Y_n\}$ generated by the Griddy Gibbs sampling method converges to a stationary distribution $\eta$ in total variation norm.  
This is a direct analog of the convergence given in Theorem~\ref{TheoGibb} above (although we still have to show that $\eta$ is near $\pi$).  

\begin{thm} (Existence and uniqueness of the invariant measure for $S$)
Assume that the approximation scheme $\{f_i\}_{i=1}^d$ satisfies condition \eqref{Con1}. Then there exists a unique probability measure $\eta$ that is invariant under $S$, and this $\eta$ satisfies
\[
\sup_{C \in B}{|Q^n(x,C)- \eta(C)| } \to 0
\]
for all $x \in D$.
In other words, $\forall x \in D$, $Q^n(x,\cdot) \to \eta(\cdot)$ in total variation norm.

\label{T1}
\end{thm}
\begin{proof}

We will prove that the transition probability $Q$ constructed in the Griddy Gibbs sampling algorithm satisfies Doeblin's condition of Theorem~\ref{C1}.  Recall that the transition probability in the Griddy Gibbs algorithm is given by
\[
Q(x,C)= \int_C{f_1(y_1,x_2,\ldots ,x_d)f_2(y_1,y_2,\ldots ,x_d)\cdots f_d(y_1,y_2,\ldots ,y_d)~dy_1 dy_2\ldots dy_d}.
\]
Recall that $f_i \geq \epsilon$ on $D$ from \eqref{Con1}.  Hence with $\text{Vol}(C)$ denoting the Lebesgue measure of $C$, we have
\[
Q(x,C) \ge \epsilon^d \text{Vol}(C), \forall x\in D, \forall C \in B.
\]
This is Doeblin's condition with $k=1$, $\phi$ is the Lebesgue measure on $D$, so applying Theorem $\ref{C1}$, we have $\sup_{C \in \mathcal{B}}{|Q^n(x,C)- \eta(C)| } \to 0$. 
\end{proof}

\subsection{Some supporting lemmas}

To establish results about regularity of the invariant measure $\eta$ of Markov chains generated by Griddy Gibbs sampling, we need the following two lemmas. 
The first result is about the absolute continuity of $\eta$, while the second result provides a basic inequality for the transition operator as a linear operator on $L^p$ space. 
As we will see in the next section, the two lemmas allow us to prove that the invariant measure $\eta$ has bounded density with respect to the Lebesgue measure.

The proofs are standard, but we sketch them for completeness.

\begin{lem}
Let ${\mu_n}$ be a sequence of probability measures on $(D,\mathcal{B})$ that converges in total variation norm to a measure $\mu$. Assume further that each ${\mu_n}$ is absolutely continuous with respect to Lebesgue measure. Then $\mu$ is also absolutely continuous w.r.t Lebesgue measure and admits a non-negative density function.
\label{L1}
\end{lem}
\begin{proof}
Consider any Borel measurable set A with $|A|=0$. By the assumption of absolute continuity, $\mu_n(A)=0$, hence $\mu(A)=\lim{\mu_n(A)}=0$. Since A was arbitrary, $\mu$ is absolutely continuous w.r.t. Lebesgue measure.
\end{proof}

Throughout the rest of this section, for any two linear normed space $U, V$, we will denote by $\mathcal{L}(U, V)$ the space of al linear operator with the source domain and the target domain being $U$ and $V$, respectively, equipped with the standard operator norm. 
We then have the following result. 

\begin{lem}
For $1 \le p \le \infty$, let $K(x,y)$ be a bounded function on $D \times D$, and let
\[
\mathcal{V}g(y)=\int{K(x,y)g(x)dx}
\]
for $g \in L^p(D)$. 
Then

\begin{itemize}

\item[(a) ] $\mathcal{V}$: $L^2(D) \to L^2(D)$ is a compact linear operator. Moreover
\[
\|\mathcal{V}\|_{\mathcal{L}(L^2 (D),L^2(D))} = \|K\|_{L^2(D \times D)}.
\]
\item[(b) ]  $\mathcal{V}$: $L^1(D) \to L^1(D)$ is a bounded linear operator. 
Moreover, if $K(x,y)$ is a transition probability function, then  
\[
\|\mathcal{V}\|_{\mathcal{L}(L^1(D),L^1(D))} \le 1.
\]
\item[(c) ] $\mathcal{V}$ maps $L^1(D)$ to $L^{\infty}(D)$, and
\[
\|\mathcal{V}\|_{\mathcal{L}(L^1(D),L^{\infty}(D))} \le \|K\|_{\infty}.
\]
\item[(d) ] If $g \in L^{2}(D)$ and $2 \le p \le \infty$ then
 \[
 \|\mathcal{V}g\|_p \le \|K\|_{p} \max \{\|g\|_1,\|g\|_2 \}. 
 \]
\end{itemize}
\label{L2}
\end{lem}

Before proceeding to provide the proof, we note that Lemma $\ref{L2}$ plays a central role in the rest of the paper. To be more precise, parts (b) and (c) of the Lemma helps establish that the invariant measure $\eta$ has bounded density, while parts (a) and (d) helps provide the $L^p$ estimates of the sensitivity of the invariant measures of Markov chains under kernel perturbation (Section 4).

\begin{proof}
Part a) of the Lemma is a well-known result about Hilbert-Schmidt integral operators. For reference, cf. \cite{P}.

For b) and c), let $M = \sup_{D \times D}{|K(x,y)|}=\|K\|_{\infty}$.  Then for all $y \in D$ we have
\[
|\mathcal{V}g(y)|=\left|\int{K(x,y)g(x)dx} \right|
\le M\int{|g(x)|dx}.
\]
In other words, $\|\mathcal{V}g\|_{\infty} \le \|K\|_{\infty}\|g\|_1$.

Integrating over $D$ gives
\[
\|\mathcal{V}g\|_1 \le \text{Vol}(D)\|\mathcal{V}g\|_{\infty} \le \text{Vol}(D)\|K\|_{\infty}\|g\|_1,
\]
which proves b) and c).  

Finally, if $K$ is a transition probability, we have
\[
\int{|\mathcal{V}g(y)|dy}=\int{\left|\int{K(x,y)g(x)dx} \right|dy}
\le\int{\int{K(x,y)dy}|g(x)|dx}
=\int{|g(x)|dx}
\]
which implies $\|\mathcal{V}\|_{\mathcal{L}(L^1 (D),L^1(D))}  \le 1$.

For d), consider the linear operator $W$ defined on $L^2(D \times D)$ and on $L^{\infty}(D \times D)$ as
\[
W\phi=\int_{D}{\phi(x,y)g(x)dx}
\]
Then from part a) and c), we have $W$ is a bounded linear operator that maps $L^2(D \times D)$ to $L^2(D)$, and maps $L^{\infty}(D \times D)$ to $L^{\infty}(D)$. Moreover, the following inequalities are satisfied:
\[
\|W\phi\|_{L^2(D)} \le \|g\|_{L^2(D)} \|\phi\|_{L^2(D \times D)}
\]
\[
\|W(\phi)\|_{L^{\infty}(D)} \le \|g\|_{L^1(D)} \|\phi\|_{L^{\infty}(D \times D)}
\]
Using Riesz-Thorin interpolation theorem (see \cite{Riesz}), we deduce that $W$ also maps $L^p(D \times D)$ to $L^p(D)$, and
\[
\|W\phi\|_{L^p(D)} \le \max\{\|g\|_{L^2(D)},\|g\|_{L^1(D)} \} \|\phi\|_{L^p(D \times D)}.
\]
Replace $\phi$ by $K$, noticing that $\mathcal{V}g=W(K)$, we deduce
 \[ 
 \|Lg\|_p \le \|K\|_{p} \max \{\|g\|_1,\|g\|_2 \}.
 \]
\end{proof}

\subsection{Regularity}

These two previous lemmas allow us to prove the following result.

\begin{thm}
(Regularity of invariant measure) The invariant measure $\eta$ of $S$ is absolutely continuous w.r.t Lebesgue measure on D. Moreover, there exists $\hat \eta \in L^{\infty}(D)$ so that for each $C \in {\mathcal B}$, 
\[\eta(C)=\int_C \hat \eta(x)dx.\]
Also, $\hat \eta$ is invariant under $S$:  $S\hat \eta=\hat \eta$.
\label{Reg}
\end{thm}
\begin{proof}
The proof of this theorem is straightforward from the previous theorems and lemma. From Theorem $\ref{T1}$ and Lemma $\ref{L1}$, we know that $\eta$ is absolutely continuous and admits a density function:
\[\eta(dx)=\hat \eta(x)dx\]
with $\hat \eta \in L^1(D)$.

Now considering $S$ as a bounded linear operator on $L^1(D)$, we have
\begin{align*}
\int_{A}{\hat \eta(x)dx}&=\eta(A)=S\eta(A)=\int_D{Q(x,A)\eta(dx)} \\
&=\int_D{\int_A{Q(x,y)dy}~\hat \eta(x)dx}=\int_A{\left(\int_D{Q(x,y)\hat \eta(x)dx}\right)~dy}.
\end{align*}
Since $A$ was arbitrary, we deduce that
\[
\hat \eta(x)=\int_D{Q(x,y)\hat \eta(x)dx}
\]
or $\hat \eta=S\hat \eta$.
From Lemma \ref{L2}, $S$ maps $L^1(D)$ to $L^{\infty}(D)$. Hence $\hat \eta=S\hat \eta \in L^{\infty}(D)$, so $\hat \eta$ is a bounded function.

\end{proof}

\begin{rem}
Since D is a subset with bounded measure of $R^d$, $\hat \eta$ also belongs to $L^p(D)$, for all $1 \le p \le \infty$.
\end{rem}

\section{Sensitivity and convergence of non-uniformly ergodic Markov chains}

Before proceeding to give result about the sensitivity of the invariant measures under perturbation, we want to make a remark that the assumption of uniformly boundedness away from 0 of the approximations $Q_i$ was introduced only to guarantee the existence and uniqueness of an absolutely continuous invariant measure $\eta$. 

As we mentioned before, we can always use additional cutoff functions on the approximation scheme to guarantee the boundedness from below of $Q_i$, without significantly affecting the accuracy of the approximation scheme. However, as an analysis of convergence of perturbed Monte Carlo Markov chains, condition( $\ref{Con1}$) is replaced by any condition that guarantees the existence and uniqueness of the invariant measure $\eta$ and the ergodicity of the Markov chain $\{Y_n\}$. In a similar manner, the assumptions of Theorem 2.1 can be replaced by the existence and uniqueness of the invariant measure $\pi$ and the ergodicity of the Markov chain $\{X_n\}$.

In short, we will assume the following conditions in the subsequent analyses

\begin{enumerate}
  \item the invariant measures $\pi$, $\eta$ of the Markov chain exists and are unique.
	\item the Markov chains $\{X_n\}$, $\{Y_n\}$ are ergodic (not necessarily uniformly ergodic).
	\item the distributions $\pi$, $\eta$ have finite second moments.	

\end{enumerate}

The distinguishing feature between our approach and other work \cite{R1,R2,R3} on convergence of perturbed Markov chain is that by considering the invariant measures as fixed points of linear operators on function spaces, we don't need to impose any further conditions on the rate of convergence of the Markov chain. For that reason, the results we derived in this paper can address the case when the considered Monte Carlo Markov chains are not uniformly ergodic.

\subsection{Continuity of eigenspaces for eigenvalue 1}
We recall from the previous part of the paper that the two transition operators $T$ and $S$ admit unique absolutely continuous invariant measures $\pi$ and $\eta$, respectively. Before proceeding to derive estimates of the distance between $\hat \pi$ and $\hat \eta$, we provide in this section two key lemmas to further investigate properties of the transition operators $T$ and $S$ as operators on $L^2(D)$. 

In Lemma $\ref{M1}$, we will prove that the eigenspaces correspond to eigenvalue $\lambda=1$ of $T$ and $S$ are one-dimensional subspaces spanned by $\pi$ and $\eta$, respectively. Lemma $\ref{M2}$ investigates a special case when it is possible to estimate the distance between the positive invariant eigenvectors of two close operators. 

\begin{lem} 
Using the same notation as in Section 2.3 and consider $T,S$ as operators on $H=L^2(D)$, we have 	
\begin{itemize}
	\item $\{v \in H: Tv=v\}=\langle \hat \pi \rangle$	
	\item $\{v \in H: Sv=v\}=\langle \hat \eta\rangle$,
\end{itemize}
where $\langle \hat \pi\rangle$ denotes the span of $\hat \pi$.
\label{M1}
\end{lem}
\begin{proof}
Consider any $w \in H- \{0\}$ such that $Tw=w$.  Then
\[
\int{|Tw(y)|dy}=\int{\left|\int{P(x,y)w(x)dx} \right|dy}
\le\int{\int{P(x,y)dy}|w(x)|dx}
=\int{|w(x)|dx}.
\]
Equality happens only when
\[
\left|\int{P(x,y)w(x)dx} \right|=\int{|P(x,y)w(x)|dx}
\]
for a.e. $y \in D$.

Since $P(x,y)>0$, this happens only if $w$ does not change sign on D. Therefore, if we define 
\[
w^*=\frac{w}{\|w\|_{L^1(D)}}
\]
then $w^*$ is the density function of a probability measure on D. Moreover, we also have $Tw^*=w^*$. Since $\pi$ is the unique invariant measure that is also a fixed point of $T$, we deduce that $w^*=\hat \pi$. Hence, $w \in \langle \hat \pi \rangle $.
\end{proof}

\begin{lem}
Let $M$ and ${N}$ be Hilbert-Schmidt integral operators on $H=L^2(D)$. Assume further that $u, v \in H$ such that
\begin{itemize}
	\item[(i) ] $\|u\|_{H}=\|v\|_{H}=1$
	\item[(ii) ] $\{w \in H: Mw=w\}=\langle u \rangle$
	\item[(iii) ] $\{w \in H: Nw=w\}=\langle v \rangle$
	\item[(iv) ] $u$, $v$ are positive functions.
\end{itemize}
Then there exists $\alpha >0$ that depends only on $M$ such that
\[
\|v-u\|_{H} \le C(\alpha)\|M-N\|_{L(H,H)}.
\]
\label{M2}
\end{lem}

\begin{proof}
Since $H$ is a Hilbert space, we can write
\[
H=\langle u \rangle \oplus K
\]
where K is the orthogonal complement of the linear space spanned by $u$. For the sake of convenience, in the rest of the proof, we will denote $\|\cdot\|_H$ simply by $\|\cdot\|$.

First we show that there exists $\alpha>0$ such that 
$$\|(M-I)k\| \ge \alpha \|k\| ~ \forall k \in K.$$ 
By way of contradiction, suppose that $\exists \alpha_n \to 0, \|k_n\|=1, k_n \in K$ such that $\|Mk_n-k_n\|=\alpha_n$. Since $M$ is a compact operator on $H$, by extracting a subsequence, we can assume that $Mk_n \to k_{\infty} \in H$. On the other hand, we have $\|Mk_n-k_n\|=\alpha_n \to 0$. By the triangle inequality, we have
\[
\|k_n-k_{\infty}\| \le \|k_n-Mk_n\| +\|Mk_n-k_{\infty}\| \to 0.
\]

We deduce that $k_n \to k_{\infty}$, and hence that $\|k_\infty\|=1$. Since K is closed we have $k_{\infty} \in K$, and since $M$ is continuous we have $Mk_{\infty}=k_{\infty}$. By (ii), $Mu=u$ has no nontrivial solution in K, so we deduce that $k_{\infty}=0$, which contradicts $\|k_\infty\|=1$.

On the other hand, we can uniquely decompose
\begin{equation} \label{vuk}
v=\lambda u + k
\end{equation}
for some $\lambda \in \RR, k\in K$.  Since $u$ and $v$ are fixed points of $M$ and $N$, respectively, we deduce that
\[
Mv=M(\lambda u + k)=\lambda Mu + Mk=\lambda u + Mk
\]
and
\[
Nv=v=\lambda u + k.
\]
Therefore
\[
\|M-N\|_{L(H,H)}\ge\|Mv-Nv\| = \|(\lambda u + Mk)-(\lambda u + k)\|=\|Mk-k\| \ge \alpha \|k\|.
\]
The orthogonal decomposition in \eqref{vuk} gives
\[
1=\|v\|^2=\lambda^2\|u\|^2 + \|k\|^2 = \lambda^2 + \|k\|^2,
\]
so
\[
\lambda^2=1-\|k\|^2 \ge 1-\left(\frac{\|M-N\|_{L(H,H)}}{\alpha} \right)^2.
\]
This plus the same decomposition also gives
\begin{align*}
\|v-u\|^2 &= (\lambda-1)^2 \|u\|^2+ \|k\|^2 = \lambda^2 -2\lambda +1 +\|k\|^2\\
 &=2(1-\lambda)= 2 \frac{1-\lambda^2}{1+\lambda}.
\end{align*}
On the other hand, from the facts that $u$, $v$ are positive functions (by (iv)) with $\|u\|=1$ (by (i)) and the orthogonal decomposition of $v$, we have $\lambda=\langle u,v\rangle = \int_{D}{uv ~ dx} \ge 0$. Hence
\[
\|v-u\|^2 \le 2(1-\lambda^2) = 2\|k\|^2 \le 2\frac{\|M-N\|_{L(H,H)}^2}{\alpha^2}
\]
or
\[
\|v-u\| \le \sqrt{2}\frac{\|M-N\|_{L(H,H)}}{\alpha}.
\]
\end{proof}

\subsection{Convergence results}

In this section, we answer the question about the sensitivity of the invariant measure of a Monte Carlo Markov chain under kernel perturbations: given that $\|P-Q\|<\epsilon$ (or equivalently, given a small perturbation on the transition operator), can we estimate the distance $\|\pi-\eta\|$ between the two invariant measures? 

The outline of this section is as follows. Using Lemma $\ref{M1}$ and $\ref{M2}$, we derive the $L^2$-estimate of the distance between $\hat \eta$ and $\hat \pi$. 

Then, knowing that S maps $L^1(D)$ to $L^{\infty}(D)$, we bound the $L^{\infty}$-norm by $L^2$-norm to produce an $L^{\infty}$-estimate, and then apply Lemma~\ref{L2} to derive the $L^p$ estimate for $2 \le p\le \infty$.

Since the proofs require us to switch back and forth between norms, let us recall that if $f \in L^{\infty}(D)$ then
\[
\|f\|_1 \le C\|f\|_2 \text{ and } \|f\|_2 \le C\|f\|_{\infty}
\]
where $C=\sqrt{\text{Vol}(D)}$.

\begin{thm}
($L^{2}$-estimate)

There exists $\delta(\pi), C(\pi)>0$ such that for $\left\|P-Q\right\|_2<\delta(\pi)$, we have
\[
\|\hat \pi-\hat \eta\|_2 \le C(\pi)\left\|P-Q\right\|_2 
\]
\label{V1}
\end{thm}

\begin{proof}

For clarity, we replace $\hat \pi$ and $\hat \eta$ with $\pi$ and $\eta$ respectively. Applying Lemma $\ref{M2}$ with $u=\frac{\pi}{\|\pi\|_2}$, $v=\frac{\eta}{\|\eta\|_2}$, we have
\begin{equation}
\left\|\frac{\pi}{\|\pi\|_2}-\frac{\eta}{\|\eta\|_2}\right\|_{2} \le \sqrt{2}\frac{\|T-S\|}{\alpha}.
\label{B1}
\end{equation}
Then
\[
\left\|\frac{\pi}{\|\pi\|_2}-\frac{\eta}{\|\eta\|_2}\right\|_{1} \le C\left\|\frac{\pi}{\|\pi\|}-\frac{\eta}{\|\eta\|}\right\|_{2}\le C \sqrt{2}\frac{\|T-S\|}{\alpha} 
\]
with $C=\sqrt{\text{Vol}(D)}.$
By the triangle inequality
\[
\left|~ \left\|\frac{\pi}{\|\pi\|_2}\right\|_{1} -\left\|\frac{\eta}{\|\eta\|_2}\right\|_{1}\right| \le C \sqrt{2}\frac{\|T-S\|}{\alpha} .
\]
Since $\pi$ and $\eta$ are probability measures, we have $\left\|\pi\right\|_1=\left\|\eta\right\|_1=1$, and hence
\[
\left|~ \frac{1}{\|\pi\|_2} -\frac{1}{\|\eta\|_2}\right| \le C\sqrt{2} \frac{\|T-S\|}{\alpha}.
\]
This leads to
		\begin{equation}
		1 -\frac{\|\pi\|_2}{\|\eta\|_2} \le C \sqrt{2}\frac{\|T-S\|}{\alpha} \|\pi\|_2.
		\label{B2}
		\end{equation}
If we assume further that the right hand side is less than 1, then 
		\begin{equation}
		\left\|\eta\right\|_2 < \frac{\|\pi\|_2}{1-C \sqrt{2}\frac{\|T-S\|}{\alpha} \|\pi\|_2}.
		\label{B3}
		\end{equation}
The triangle inequality plus \eqref{B1} and \eqref{B2} give
\begin{align*}
\|\pi-\eta\|_2 &\le \left\|~ \pi -\frac{\|\pi\|_2 \eta}{\|\eta\|_2}\right\|_2 + \left\|~ \eta -\frac{\|\pi\|_2 \eta}{\|\eta\|_2}\right\|_2 \\
&\le \sqrt{2}\frac{\|T-S\|}{\alpha} \|\pi\|_2 + C\sqrt{2}\frac{\|T-S\|}{\alpha}\|\pi\|_2 \|\eta\|_2
\end{align*}
and then \eqref{B3} gives
\begin{equation} \label{piDiffEta}
\|\pi-\eta\|_2 \le \sqrt{2}\frac{\|T-S\|}{\alpha} \|\pi\|_2 \left( 1 + C \frac{\|\pi\|_2}{1-C\sqrt{2} \frac{\|T-S\|}{\alpha}\|\pi\|_2 }\right).
\end{equation}
Since $T$ is defined by $\pi$, we can consider $\alpha$ as a function of $\pi$ only. Moreover, with $\delta(\pi)= \alpha/(2C\sqrt{2}\left\|\pi\right\|_2)$ and $\|T-S\| \leq \delta(\pi)$, the right hand side of \eqref{B2} is at most $1/2$, and the constant in parentheses in \eqref{piDiffEta} is at most $1+2C\|\pi\|_2$.  Hence we define 
\[
C(\pi)= \frac{\sqrt{2}\left\|\pi\right\|_2\left(1+2C\left\|\pi\right\|_2\right)}{\alpha}
\]
and note that from Lemma $\ref{L2}$,
 \[
 \|T-S\|_{\mathcal{L}(L^2,L^2)} = \|P-Q\|_{L^2(D \times D)}.
 \]
Hence for $\left\|P-Q\right\|_2 <\delta(\pi)$, changing back to the original notations, we have the desired estimate
\[\|\hat \pi-\hat \eta\|_2 \le C(\pi)\left\|P-Q\right\|_2. \]

\end{proof}

\begin{rem}
From \eqref{B3} and the choice of $\delta(\pi)$, we see that if $\left\|P-Q\right\|_2 <\delta(\pi)$, then $\|\hat \eta\|_2<2\|\hat \pi\|_2$.
\label{Repieta}
\end{rem}

\begin{thm}
($L^{\infty}$-estimate)

There exists $\delta'(\pi), C'(\pi)>0$ such that if $P,Q \in L^{\infty}(D \times D)$ and  $\left\|P-Q\right\|_{\infty}<\delta'(\pi)$ then
\[
\|\hat \pi-\hat \eta\|_{\infty} \le C'(\pi)\left\|P-Q\right\|_{\infty}.
\]
\label{V2}
\end{thm}

\begin{proof}
As in the proof of the previous theorem, we replace $\hat \pi$ and $\hat \eta$ with $\pi$ and $\eta$ respectively.  
Using part the fact that $\pi$ and $\eta$ are fixed by $T$ and $S$, respectively, plus the triangle inequality and c) of Lemma $\ref{L2}$, we have
\begin{eqnarray}
\|\eta-\pi\|_{\infty} = \|S\eta-T\pi\|_{\infty} &\le& \|S\eta-T\eta\|_{\infty} + \|T\eta-T\pi\|_{\infty} \nonumber\\
&\le& \|P-Q\|_{\infty}\|\eta\|_1  + \|P\|_{\infty} \|\eta-\pi\|_1 \nonumber\\
&\le& C\|P-Q\|_{\infty}\|\eta\|_2  + C\|P\|_{\infty} \|\eta-\pi\|_2, \label{AA}
\end{eqnarray}
with $C=\sqrt{\text{Vol}(D)}$. With $\delta(\pi)$ and $C(\pi)$ as in the previous theorem, define
\[
\delta'(\pi)=\frac{\delta(\pi)}{C} \ \ \text{ and } \ \ C'(\pi)=2C\|\pi\|_2 + C^2\|P\|_{\infty} C(\pi).
\]
If $\|P-Q\|_{\infty}<\delta'(\pi)$, then as mentioned previously,
\[
\|P-Q\|_{2} \le C \|P-Q\|_{\infty}<\delta(\pi).
\]
We start with $\ref{AA}$ and then use $\|\eta\|_2<2\|\pi\|_2$ and $\|\pi-\eta\|_2 \le C(\pi)\left\|P-Q\right\|_2$ from remark $\ref{Repieta}$ and theorem $\ref{V1}$ to get
\begin{align*}
\|\eta-\pi\|_{\infty} &\le C\|P-Q\|_{\infty}\|\eta\|_2  + C\|P\|_{\infty} \|\eta-\pi\|_2\\
&\le 2C\|P-Q\|_{\infty}\|\pi\|_2 + C\|P\|_{\infty} C(\pi) \|P-Q\|_2\\
\end{align*}
By collecting terms and noticing that $\|P-Q\|_2 \le C\|P-Q\|_{\infty}$, we deduce that
\begin{align*}
\|\eta-\pi\|_{\infty} 
&\le \left(2C\|\pi\|_2 + C^2\|P\|_{\infty} C(\pi)\right) \|P-Q\|_{\infty}\\
&=C'(\pi)\|P-Q\|_{\infty}
\end{align*}
\end{proof}

\begin{thm}
($L^{p}$-estimate, $2 \le p \le \infty$)

Let $2 \le p \le \infty$, there exists $\delta'(\pi), C'(\pi)>0$ such that if $P,Q \in L^p(D \times D)$ and  $\left\|P-Q\right\|_{p}<\delta'(\pi)$ then
\[
\|\hat \pi-\hat \eta\|_{p} \le C'(\pi)\left\|P-Q\right\|_{p}.
\]
\label{V3}
\end{thm}

\begin{proof}
As before, we replace $\hat \pi$ and $\hat \eta$ with $\pi$ and $\eta$ respectively. 
Applying Lemma $\ref{L2}$ (part d), noticing that $ \eta$ and $\pi$ belong to $L^{2}(D)$, we have:
	\[
	\|S\eta-T\eta\|_{p} \le \|P-Q\|_{p} \max \{\|\eta\|_1,\|\eta\|_2\}
	\]
and
	\[
	\|T\eta-T\pi\|_{p} \le \|P\|_{p} \max \{\|\eta-\pi\|_1,\|\eta-\pi\|_2 \}.
	\]
Using the fact that $\pi$ and $\eta$ are fixed by $T$ and $S$, respectively, plus the triangle inequality and c) of Lemma $\ref{L2}$, we have
\begin{eqnarray}
\|\eta-\pi\|_{p} = \|S\eta-T\pi\|_{p} &\le& \|S\eta-T\eta\|_{p} + \|T\eta-T\pi\|_{p} \nonumber\\
&\le& \|P-Q\|_{p}\max \{\|\eta\|_1,\|\eta\|_2\}  + \|P\|_{p}\max \{\|\eta-\pi\|_1,\|\eta-\pi\|_2\} \nonumber\\
&\le& C\|P-Q\|_{p}\|\eta\|_2  + C\|P\|_{p} \|\eta-\pi\|_2, \label{BB}
\end{eqnarray}
with $C=\sqrt{\text{Vol}(D)}$. 
The rest of the proof concludes as in the proof of the previous theorem.
\end{proof}

\section{Extension to non-compact support distributions}

While most of the assumption of the method on the ergodicity of the Markov chains are quite general, one restriction of the method comes from the assumption of bounded parameter space $D$. Since the key ideas of our analysis of sensitivity of the invariant measures rely on moving back and forth between the $L^p$-norms, this condition could not be easily removed from the framework.

However, it is worth noting that for distributions with non-compact support, a variation of the Griddy Gibbs sampling method can be developed as followed: first, a rectangular domain $D$ is chosen by prior knowledge about $\pi$, then the Griddy Gibbs sampling with $\pi_{new}= \pi|_D$ (normalized by a constant) is proceeded as usual. By our previous analyses, the Monte Carlo Markov chains generated by this process will have a unique invariant measure $\eta$ whose distance to $\pi$ can be estimated by the following theorem

\begin{thm}

Let $2 \le p \le \infty$. Assume that $\pi$ has non-compact support on $\mathcal{R}^d$ and that there exist $C_1, C_2>0$ such that:
\[
\int_{\mathcal{R}^d}{\|x\|_1 \hat \pi(x)^p~dx} ~\le C_1 \h \text{and} \h 
\int_{\mathcal{R}^d}{\|x\|_1 \hat \pi(x)~dx} ~\le C_2
\]
where $\|x\|_1= |x_1|+...+|x_d|$.
Let $D_t=\{x\in \mathcal{R}^d: \|x\|_{\infty}>t\}$ where $\|x\|_{\infty} = \max_{i}{|x_i|}$.

There exists $\delta'(\pi), C'(\pi)>0$ such that if $P,Q \in L^{p}(D \times D)$, $\left\|P-Q\right\|_{p}<\delta'(\pi)$ and $t \ge C_2/2$ then
\begin{equation}
\|\hat \pi-\hat \eta\|_{p} \le C'(\pi,D_t)\left\|P-Q\right\|_{p}+ \frac{C_2}{2t}\|\hat \pi\|_p + \frac{C_1}{t\|\hat \pi\|_p}
\label{star}
\end{equation}
\label{V4}
\end{thm}

\begin{proof}

We denote $f = \hat \pi^p/ \|\hat \pi\|_p$ and let $X$ be a random variable with density function $f$. 
By Markov's inequality, we have for $i=1,2,...d$
\begin{eqnarray*}
\mathbb{P}[|X_i| > t] &\le& \frac{1}{t} \int_{\mathcal{R}^d}{|x_i|f(x)~dx} \\
&=& \frac{1}{t\|\hat \pi\|_p} \int_{\mathcal{R}^d}{|x_i|\hat \pi^p(x)~dx}
\end{eqnarray*}
Hence
\[
\left\|\hat \pi\right\|_{L^p(\mathcal{R}^d \setminus D_t)}=\mathbb{P}[\|X\|_{\infty} > t] \le \frac{\int_{\mathcal{R}^d}{\|x\|_1 \hat \pi^p(x)~dx}}{t\|\hat \pi\|_p} 
\]
By a similar argument, we have
\[
\left\|\hat \pi\right\|_{L^1(\mathcal{R}^d \setminus D_t)}\le \frac{1}{t}\int_{\mathcal{R}^d}{\|x\|_1 \hat \pi(x)~dx} \le \frac{C_2}{t}
\]
On the other hand, results for distribution with compact support in $D$ implies
\[
\left\|\hat \eta - \frac{\hat \pi}{\int_{D_t}{\hat\pi}}\right\|_{L^p(D)} \le  C'(\pi,D_t)\|P-Q\|_p,
\]
We deduce that, for $t \ge 2C_2$, we have
\begin{eqnarray*}
\|\hat \pi-\hat \eta\|_{L^p(\mathcal{R}^d)} &\le& \|\hat \pi-\hat \eta\|_{L^p(D)} + \left\|\hat \pi\right\|_{L^p(\mathcal{R}^d \setminus D_t)} \\
&\le& \left\|\hat \eta - \frac{\hat \pi}{\int_{D_t}{\hat\pi(x)~dx}}\right\|_p +\frac{\int_{\mathcal{R}^d \setminus D_t}{\hat\pi(x)~dx}}{\int_{D_t}{\hat\pi}} \|\hat \pi\|_p+ \left\|\hat \pi\right\|_{L^p(\mathcal{R}^d \setminus D_t)}\\
&\le& C'(\pi,D_t)\left\|P-Q\right\|_{p}+ \frac{C_2}{2t}\|\hat \pi\|_p + \left\|\hat \pi\right\|_{L^p(\mathcal{R}^d \setminus D_t)}\\ 
&\le& C'(\pi,D_t)\left\|P-Q\right\|_{p}+ \frac{C_2}{2t}\|\hat \pi\|_p + \frac{C_1}{t\|\hat \pi\|_p}
\end{eqnarray*}

\end{proof}

\begin{cor}
Let $2\le p \le \infty$ and assume that $\pi$ has non-compact support on $\mathcal{R}^d$ and that there exists $C_3, C_4>0$ so that:
\begin{enumerate}
\item $\int_{\mathcal{R}^d}{|x|^2 \hat \pi(x)~dx} \le C_3 < \infty$
and
\item $\|\hat\pi\|_{L^{2p-1}} \le C_4 < \infty$
\end{enumerate}
Then result (\ref{star}) is true  with $C_1=\sqrt{C_3C_4^p}$ and $C_2=1+C_4$.
\end{cor}
That is, if prior estimates on the second moment of $\pi$ and the $L^{2p-1}$-norm of $\hat \pi$ are available, the Griddy Gibbs algorithm can be adjusted accordingly to produce a good estimate on the distance between the two invariant measures.

\begin{proof}
By Holder's inequality with 
\[
u(x)=\|x\|_1\hat\pi(x)^{1/2}, ~~~ v(x)=\hat\pi^{p-1/2}
\]
we have
\begin{eqnarray*}
\int_{\mathcal{R}^d}{\|x\|_1 \hat \pi^p(x)~dx} &\le& \left(\int_{\mathcal{R}^d}{\|x\|_1^2 \hat \pi(x)~dx}\right)^{1/2}\left(\int_{\mathcal{R}^d}{\hat \pi^{2p-1}(x)~dx}\right)^{1/2} \\
&\le& \sqrt{C_1 C_2^p}.
\end{eqnarray*}
and
\[
\int_{\mathcal{R}^d}{\|x\|_1 \hat \pi(x)~dx} ~\le 
\int_{\|x\|_1 \le 1}{\hat \pi(x)~dx} + 
\int_{\|x\|_1 > 1}{\|x\|^2_1 \hat \pi(x)~dx} \le 1+ C_4
\]
\end{proof}

\section{Numerical examples}
In this section, we provide numerical examples to illustrate our theoretical findings and demonstrate the utility of the Griddy Gibbs sampling method. First, we validate the estimates derived in previous sections in a simple 2D example. We then proceed to investigate the performance of the Griddy Gibbs sampling in a practical example arising from systems biology, in which it is necessary to employ the Griddy Gibbs sampling method, and demonstrate the use of the method in making inferences about the system.

\subsection{A 2D example}

In this example, we investigate the performance of the Griddy Gibbs sampling algorithm on grids of various resolutions in a simple 2D example. The chosen distribution for has the following density function
\[
\pi(x,y)=\frac{1}{2}Beta \left(\frac{x+1}{2},2,5\right)*Beta \left(\frac{y+1}{2},2,5 \right)
\]
\[
+~ \frac{1}{2}Beta \left (\frac{x+1}{2},2,2 \right)*Beta \left(\frac{y+1}{2},2,2 \right)
\]
where $Beta(x,\alpha,\beta)$ is the density of the one-dimensional Beta distribution with parameter $\alpha$ and $\beta$. 

This distribution was chosen specifically to illustrate the developed framework in the case of compact support: it has compact support in its domain $[-1,1] \times [-1,1]$ and has non-independent components, but the 1D marginal density functions can be obtained in simple form:
\[
\pi_X(x) = \frac{1}{2}Beta \left(\frac{x+1}{2},2,5\right) + \frac{1}{2}Beta \left (\frac{x+1}{2},2,2 \right).
\]
\begin{figure}
\centering
\mbox{\subfigure{\includegraphics[width=0.45\textwidth]{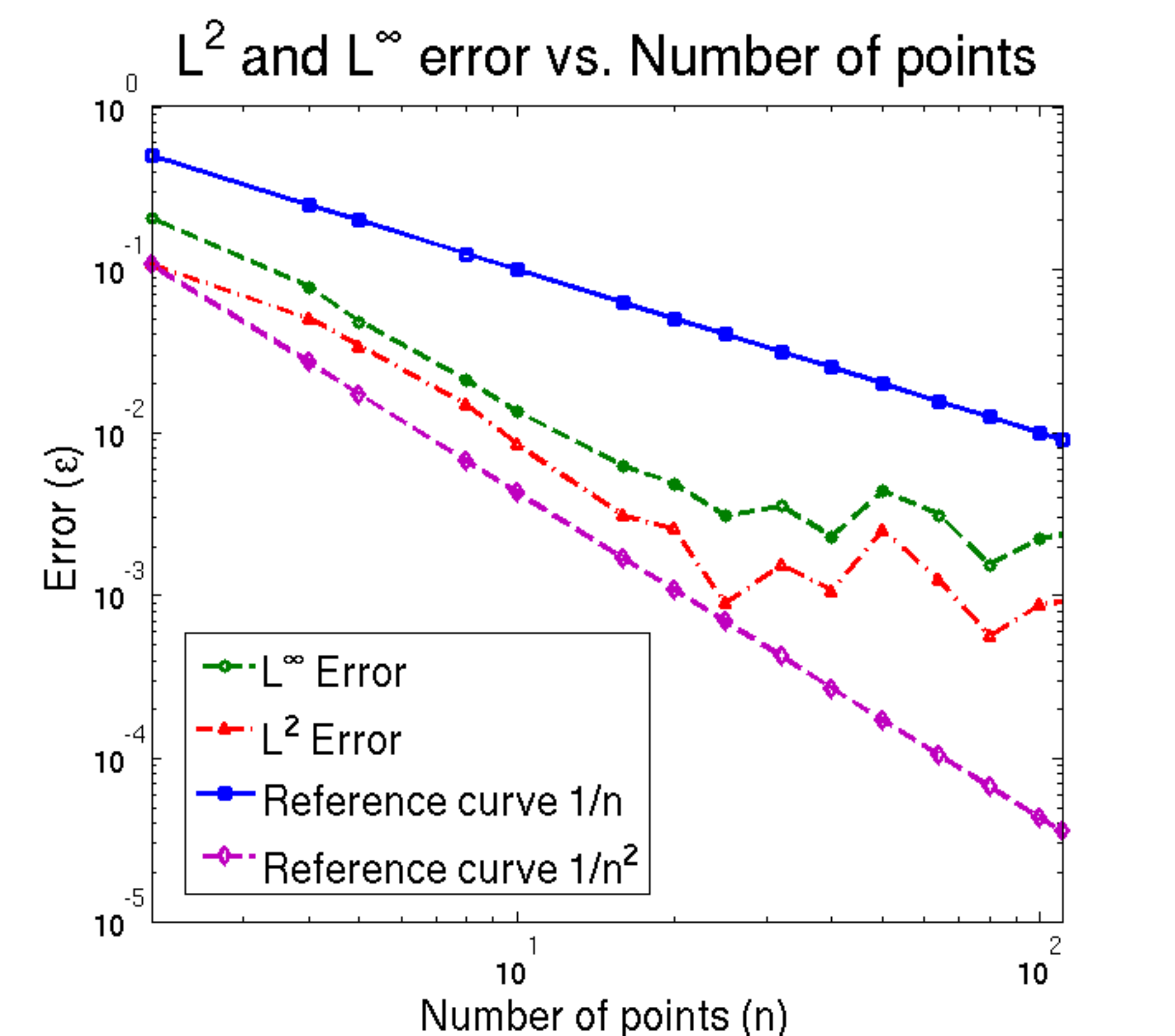}}\quad
\subfigure{\includegraphics[width=0.45\textwidth]{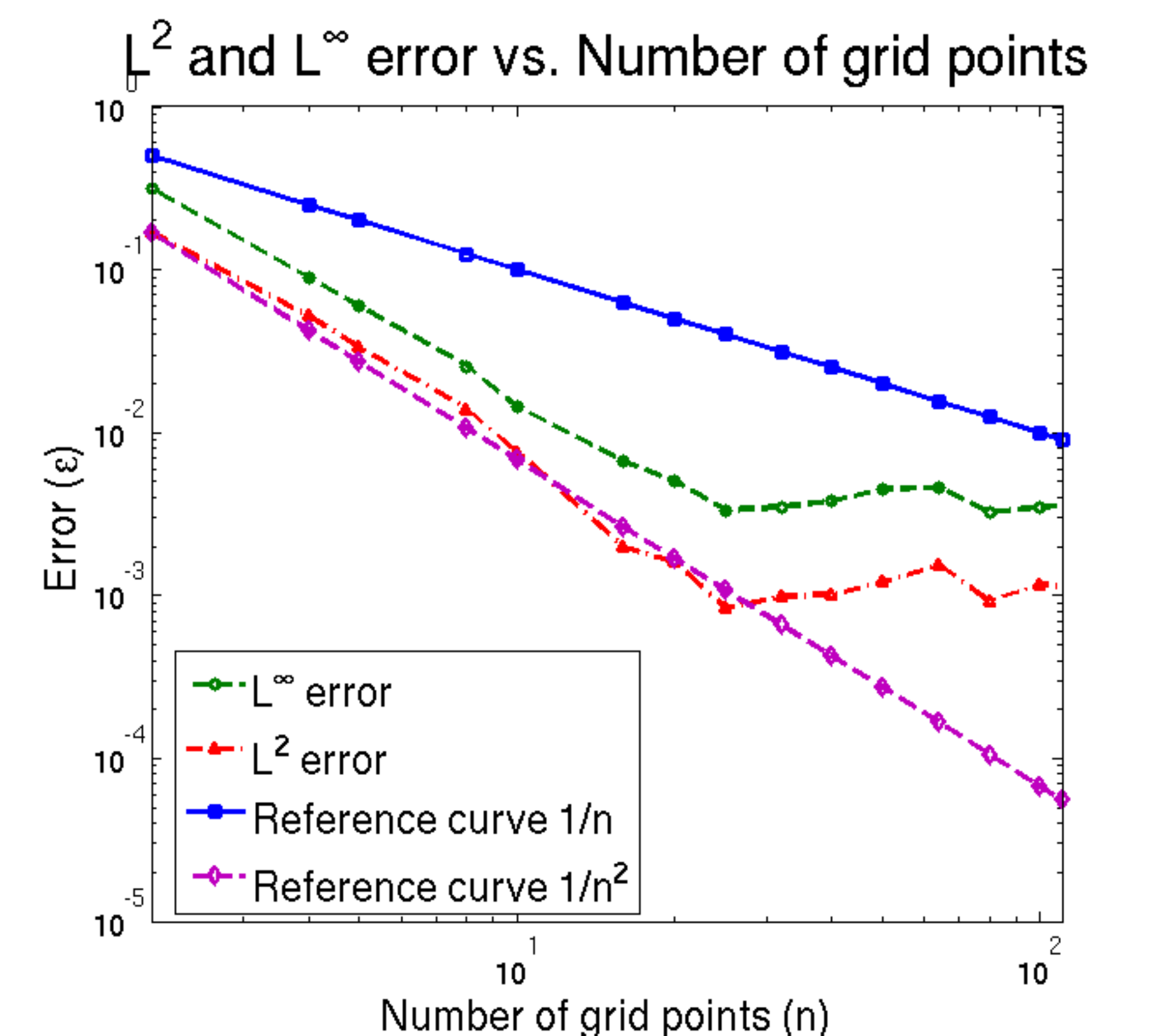} }}
\caption{\footnotesize Left: Error of the 1D marginal empirical cumulative distribution function, and Right: error of the empirical cumulative distribution function, both as a function of the number of points used in the approximation grid.}
\end{figure}

Using this probability distribution, we illustrate the estimates provided in previous sections, by expressing the $L^2$ and $L^{\infty}$ distance between the estimator (using Griddy Gibbs) and the true distribution of interest in terms of the number of points used in the grid of approximation.
For various number number of points $n$ used in the grid of approximation, we use standard linear interpolation on $n$ equally spaced points in the interval $[-1, 1]$ to approximate the 1-dimensional conditional distributions. 
For each grid, a Griddy Gibbs chain of length $10^5$ is generated.
We then use the sampled points to estimate the empirical cumulative distribution function (ECDF) and the 1D marginal ECDF of the invariant distribution of the chains. 
Finally, the $L^2$ and $L^{\infty}$ distance between the estimated ECDFs with different number of grid points and the true CDF are calculated. 

We note that in the context of our example, it is more convenient to work with CDFs rather than with PDFs for two main reasons: (i) CDFs can be approximated using nonparametric estimators; and (ii) there is a well-developed theoretical machinery for the comparison of CDFs using such estimators. 
Moreover, it is well-known that the ECDF is a non-parametric, unbiased estimator that converges uniformly to the true CDF (a result known as the Glivenko - Cantelli theorem \cite{LK}).

The results are illustrated in Figure 2. The error of both ECDF and the marginal ECDF of the first variable decrease faster than $O(\frac{1}{n})$ and approximately as fast as $O(\frac{1}{n^2})$ when the number of the grids point $n$ increases, until it reaches a level at which the error of the Griddy Gibbs sampling is dominated by the error of the Monte Carlo simulation. Since the accuracy of standard 1D linear approximation method is bounded by $O(\frac{1}{n})$, and can be as fast as $O(\frac{1}{n^2})$ if the function has bounded second derivative, this confirms our theoretical results about linear dependency between error of the 1D approximation, and the distance from the estimated distribution to the true distribution of interest.

\subsection{An example in systems biology.}

In this example, we consider a mathematical model of the T-cell signaling pathway proposed by Lipniacki et al. in \cite{LB}. The behaviour of the system is modelled as an ODE system controlled by 19 different parameters with 37 state variables and fixed initial conditions:  
\begin{align*}
\dot{x}&=\alpha(\omega,x) \mbox{\hspace{.3in}} \text{(System of ODEs)}\\
x(0)&=x_0(\omega) \mbox{\hspace{.4in}} \text{(Initial conditions)}\\
y(t)=f(\omega,t)&=\beta(\omega,x(t)) \mbox{\hspace{.3in}} \text{(Output)}
\end{align*}
Here $x=(x_1,x_2,...,x_{n_x})\in M \subset \RR^{n_x}$ is the state variable, with M a subset of $\RR^{n_x}$ containing the initial state, and $f(\omega,t) \in \RR$ is the output response (system dynamics). In the scope of this paper, we are interested in the dynamics of pZap, one of the state variables of the system. The vector of unknown parameters is denoted by $\omega=(\omega_1,...,\omega_N) \in \RR^N$ and is assumed to belong to a subset $\Omega$ of $\RR^N$. These functions and initial conditions depend on the parameter vector $\omega \in \Omega$.

The traditional approach to study such a system is to estimate values of the parameters from observations. However, in the field of systems biology, usually it is not possible to estimate all parameters in a given model, in particular if the model is complex and the data is sparse and noisy. Thus, to represent explicitly the state of knowledge, it is best to consider not a single parameter valuation but the whole space of uncertain parameters. The uncertainty in parameter values is often characterized by a probability distribution $\pi(\omega)$ on the set of all possible parameter values, based on how the output of the system driven by a particular parameter valuation fits previous data. This gives a distribution with density
\begin{equation}  \label{EDEDist}
\hat \pi(\omega)=c_n ~ \exp\left(-\sum_{i=1}^n{\left|f(\omega,t_i)-d(t_i)\right|^2}\right),
\end{equation}
where $c_n$ is a normalizing constant, $(t_1,d_1),...(t_n,d_n)$ is the set of previous data.

Inference about the system will be made based on $\pi$. For example, in \cite{D,M}, the optimal experiment is chosen at the time point where the maximum value of the normalized variance of the outputs with respect to $\pi$ is achieved. Another example was given in \cite{DRB1} where the expected dynamics estimator to recover the correct system dynamics is defined as the expected value of the system dynamics with respect to the distribution $\pi$. 

This motivates the problem of sampling with respect to the distribution $\pi$. As noted in the introduction, the use of the standard Gibbs sampling method is hindered by two factors: first, there is no closed-form formula for the distribution $\pi$ or for the corresponding one-dimensional conditional distributions; second, the evaluation of the unnormalized distribution at one point is computationally expensive (it is equivalent to solving a high dimensional system of differential equations). It is then necessary to approximate the conditional distribution by functions of simpler forms. The Griddy Gibbs method therefore is a suitable choice for this sampling process.

In this particular example, we restrict the analysis to the five most sensitive parameters with respect to perturbation. This choice is based on previous knowledge about the dynamics of the system and on the result of a global sensitivity analysis using sparse grid interpolation(\cite{GTB}).

To further reduce the computational cost, we also employ a sparse grid interpolant to approximate the output of the ODE system. That is, the output functions of the system of ODEs are evaluated on a sparse grids of $10^5$ points on the parameter space, then the method of sparse grid interpolation is employed to approximate the outputs at other sets of parameter values. Moreover, the one-dimensional conditional distributions are then approximated by piecewise linear functions on grids of fineness $\delta=0.2$ (which corresponds to a grid with 11 equally spaced points). It is worth noting that although this is a two-leveled approximation, it still fits into the framework developed in previous sections.

We will compare the performance of the Griddy Gibbs sampling with the variation of Gibbs sampling suggested by Tierney et al. in \cite{T}. In Tierney's algorithm, a Metropolis chain is embedded to ensure that the equilibrium distribution is exactly $\pi$ even on a coarse grid. The drawback is that the computational cost is at least twice as much as Griddy Gibbs sampling using the same grid. Moreover, the algorithm is more difficult to set up and is restricted to piecewise linear and piecewise constant approximations. 

In Figure~\ref{conditional}, we use samples from Griddy Gibbs and from Tierney's algorithm  to compare the conditional and marginal distribution derived from the ECDFs. In the left panel, we compare the conditional distributions (using samples from Griddy Gibbs and Tierney's algorithm) of the second parameter on the first paramter, for various values of this parameter. In the right panel, the difference between the two marginal joint distributions of the first and the second variable are computed. We also compare the difference between the two marginal joint distributions of the first and the second variable while using various numbers of samples in Figure~\ref{SqrtN}. The results from Figure~\ref{conditional} and Figure~\ref{SqrtN} suggest that the Griddy Gibbs sampling method is as effective as  Tierney's algorithm (whose convergence is also guaranteed theoretically) in generating Markov chains with with respect to a given invariant measure: the difference between the two marginal distributions is of the same magnitude as the error of the Monte Carlo method itself. 

We then investigate the performance of the Griddy Gibbs sampling in making inferences about dynamics. For this we consider the Expected Dynamics Estimator based on one single simulated data point.  This generates a distribution $\pi_1$ as in \eqref{EDEDist} with $n=1$, and we then use this distribution to estimate the system dynamics by
\[
\hat{D}_1(t)=E_{\pi_1(\omega)}[f(\omega,t)].
\]
The results are provided in Figure~\ref{EDE} (Left). The expected dynamics are calculated using the empirical mean of the output values on the previous two sets of samples. Once again, the performance of the Griddy Gibbs sampling is as good as Tierney's algorithm in computing the expected dynamics. 

Finally, Figure~\ref{EDE} (Right) compares the auto-correlation coefficients of the Monte Carlo Markov chains generated by the two algorithms. To compute the auto-correlation coefficients, two Monte-Carlo Markov chains of length $10^5$ were generated by the two algorithms, respectively. The figure illustrates the fact that not only is the computational cost of Tierney's algorithm (to generate one instance of the chain) higher, but also its auto-correlation function converges (to zero) at a much lower rate. In this particular example, if one wants to get two sets of i.i.d samples with the same number of points by both algorithms, the computational cost for Tierney's algorithm is at least ten times that of the cost for Griddy Gibbs. 

\begin{figure}
\centering
\mbox{\subfigure{\includegraphics[width=0.45\textwidth]{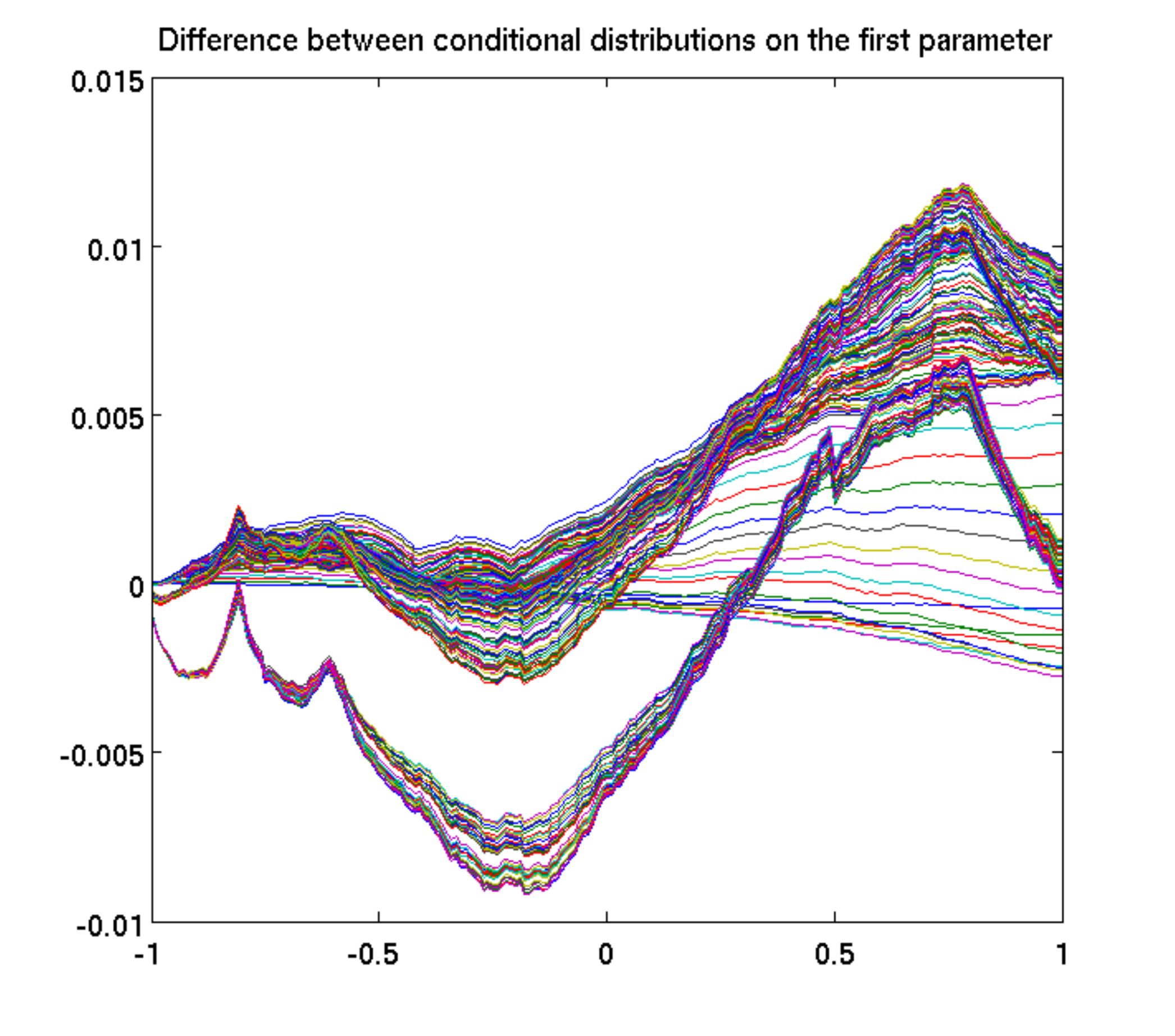}}\quad
\subfigure{\includegraphics[width=0.45\textwidth]{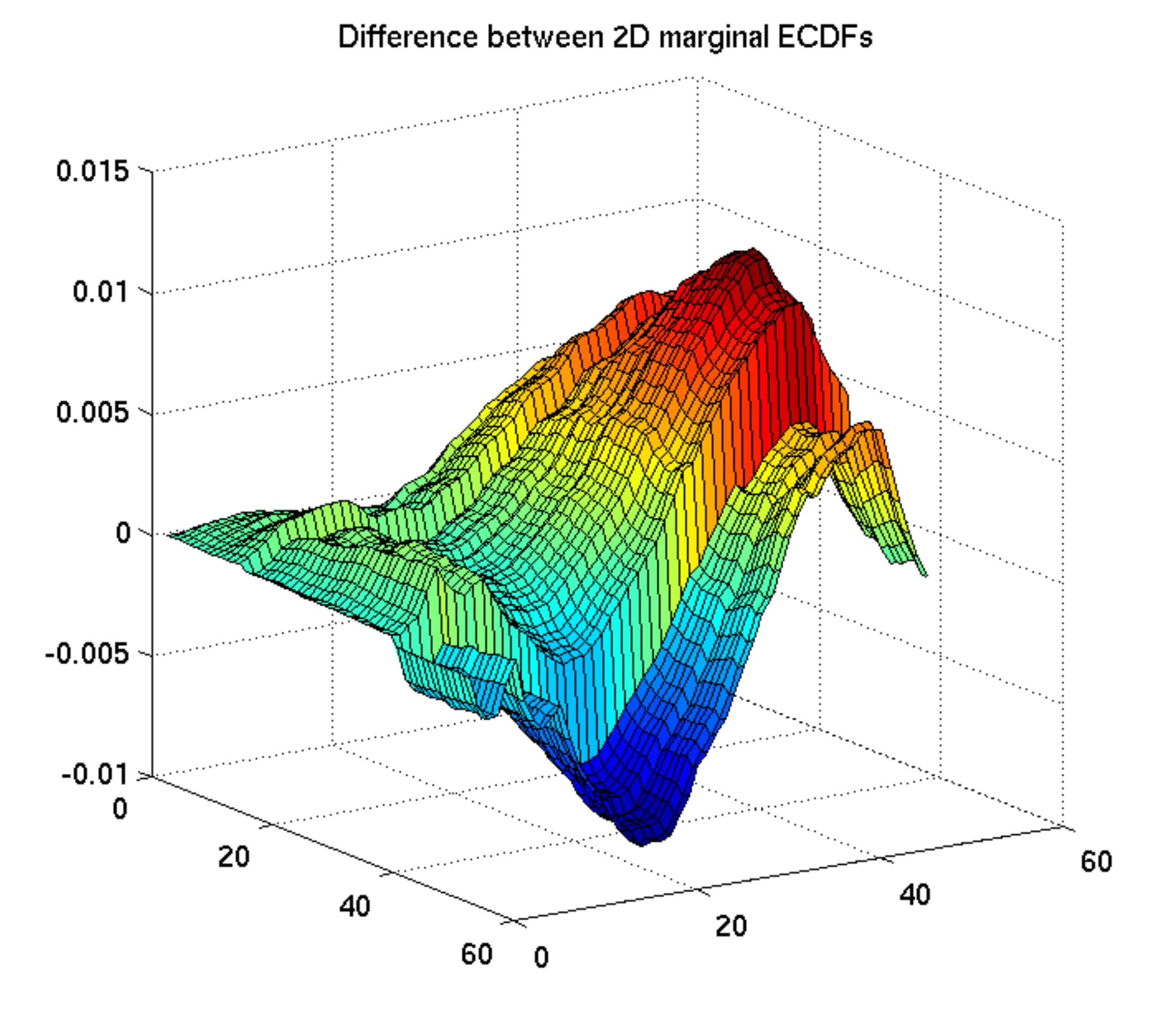} }}
\caption{\footnotesize Conditional and marginal distribution for the T-cell model. Left: The difference between the conditional distributions on the first parameters (one curve for each value of this parameter). Right: The difference between the marginal joint distributions of the first two parameters, achieved from Griddy Gibbs and Tierney's algorithm. Figure 4 shows that the differences between corresponding ECDFs are of the same magnitude as the error of the Monte Carlo method ($O(\frac{1}{\sqrt{N}})$, where N is the number of samples)}
\label{conditional}
\end{figure}

\begin{figure}
\centering
\includegraphics[width=0.7 \textwidth]{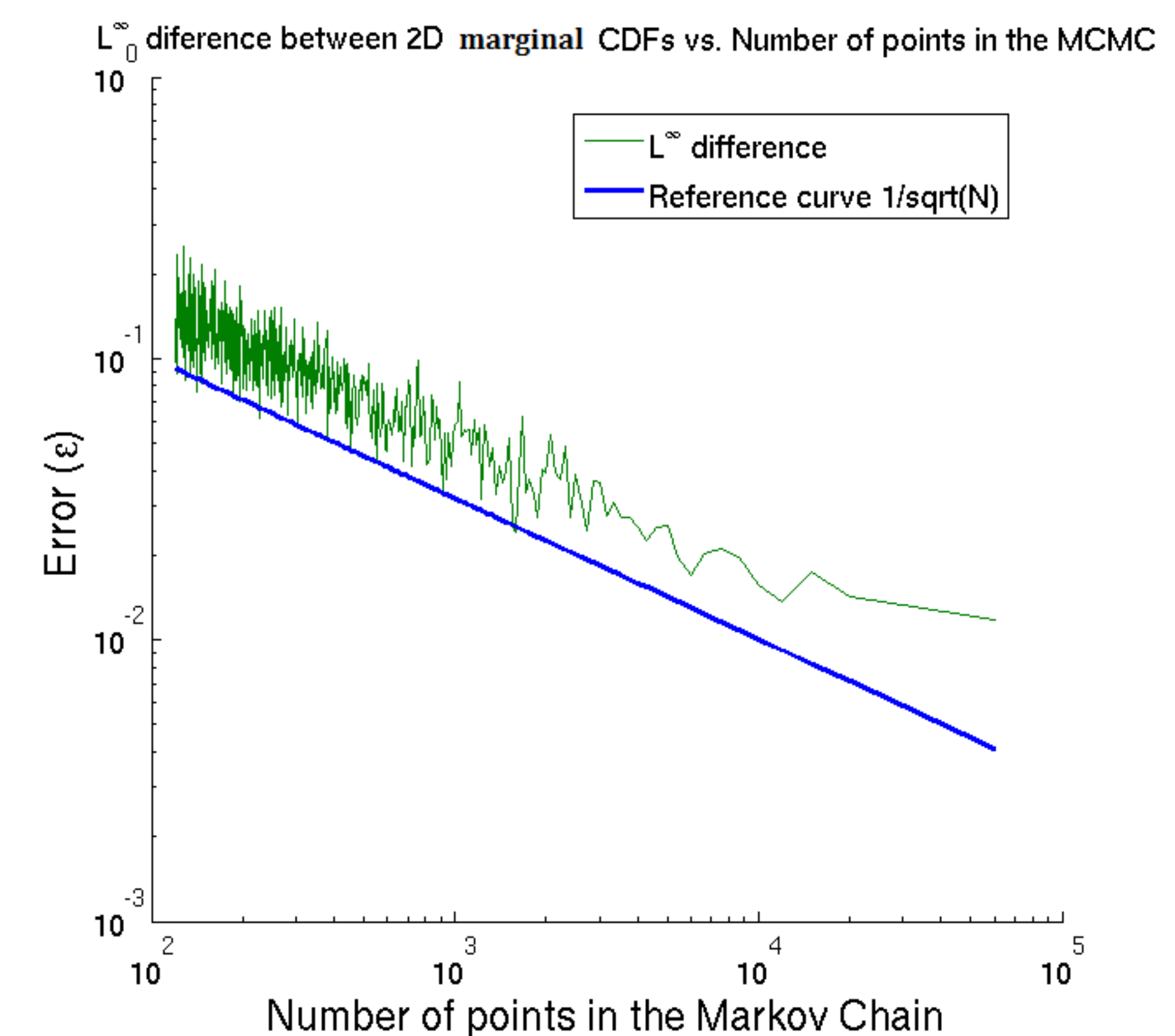}
\caption{The difference between the marginal distributions computed by Griddy Gibbs and Tierney's algorithm is of the same magnitude as the error of the Monte Carlo method itself ($O(\frac{1}{\sqrt{N}})$, where N is the number of samples).}
\label{SqrtN}
\end{figure}

\begin{figure}
\centering
\mbox{\subfigure{\includegraphics[width=0.45\textwidth]{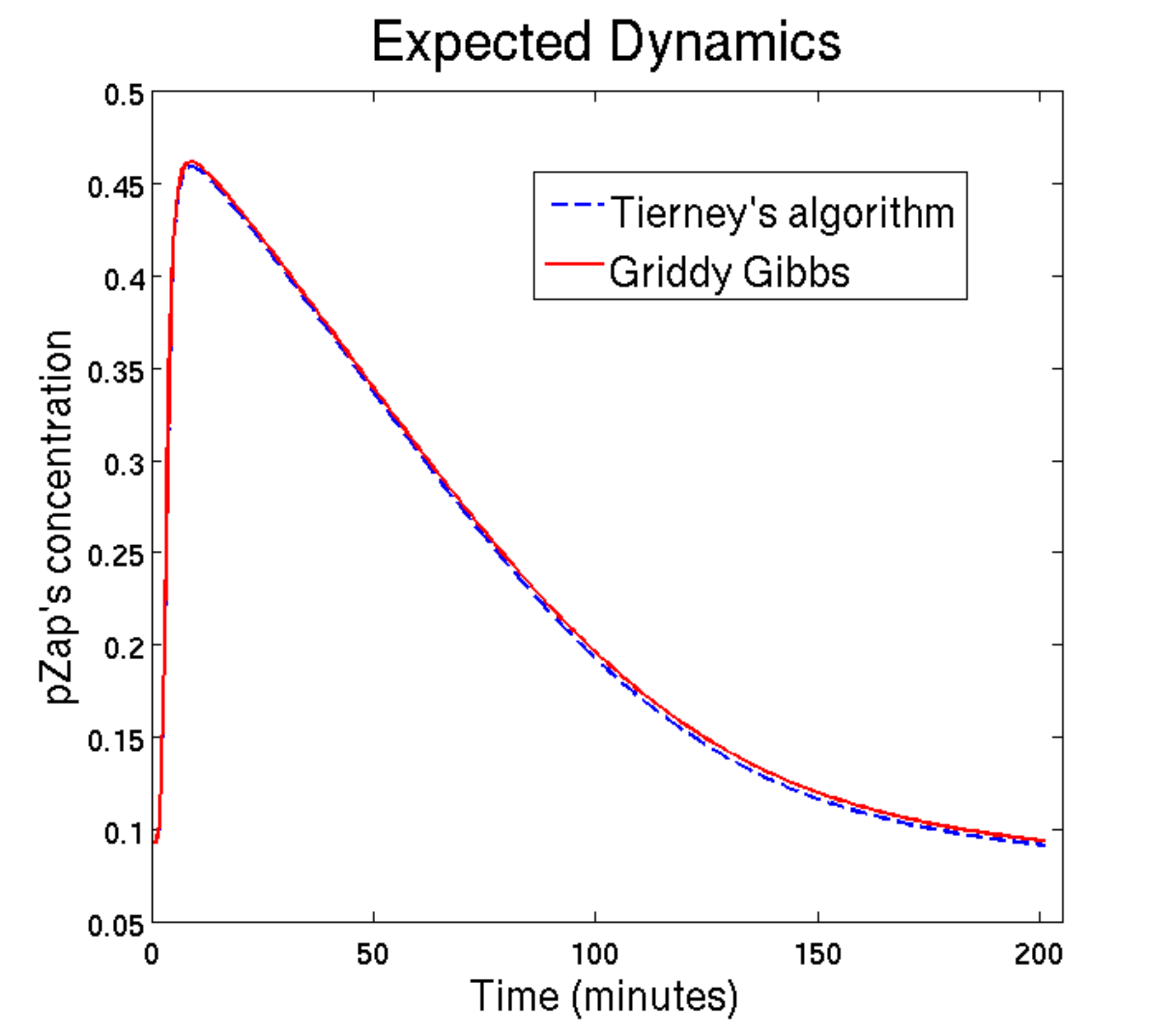}}\quad
\subfigure{\includegraphics[width=0.45\textwidth]{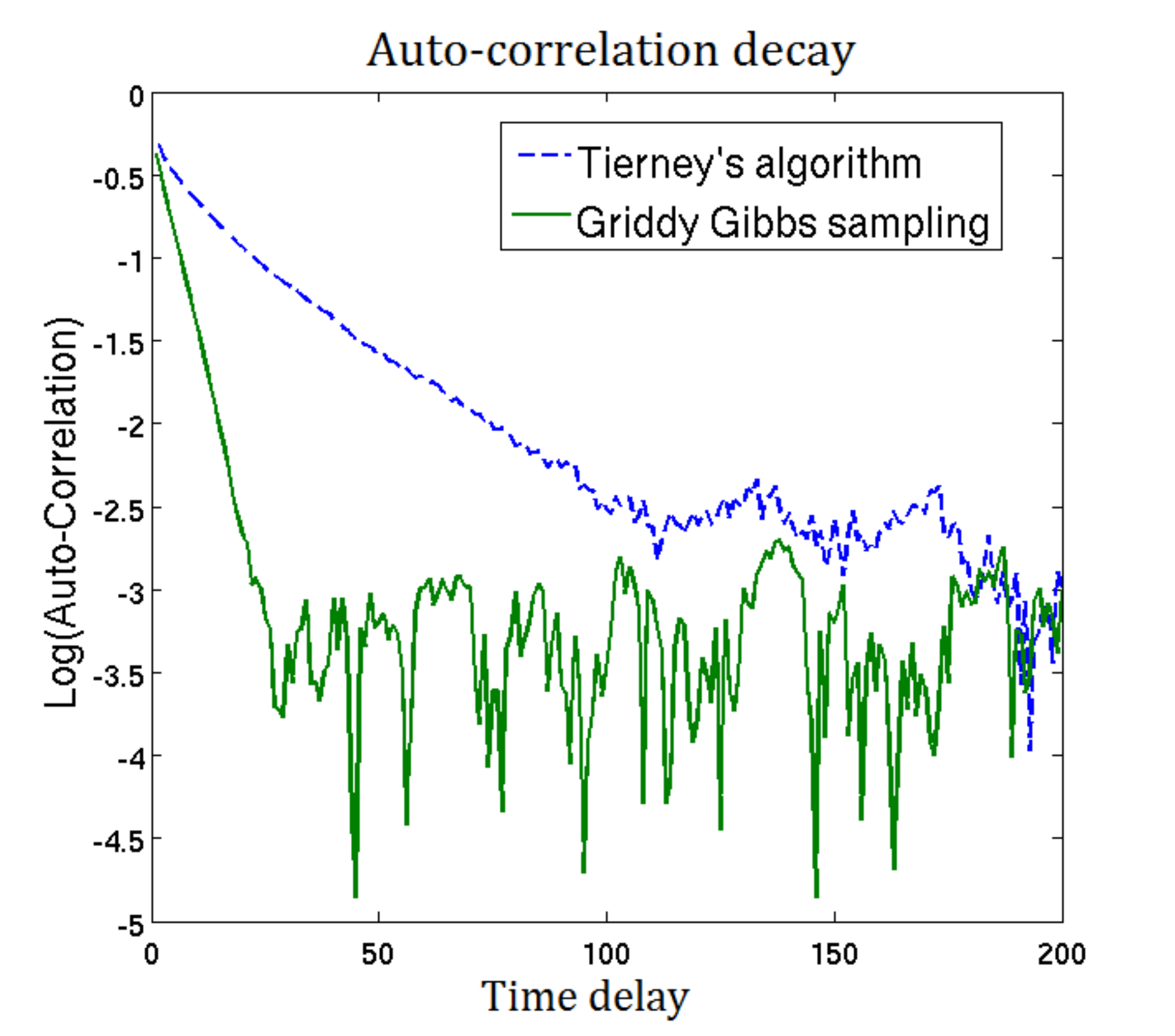} }}
\caption{\footnotesize Left: The expected dynamics estimator based on one data point, generated by Griddy Gibbs and Tierney's samples. Right: Auto-Correlation coefficients of the Markov chains generated by Griddy Gibbs algorithm and Tierney's algorithm.}
\label{EDE}
\end{figure}

\section{Conclusion}

We have shown, subject to some fairly natural conditions, that the Griddy Gibbs method has a unique, invariant measure.  Moreover, we gave $L^p$ estimates on the distance between this invariant measure and the corresponding measure obtained from Gibbs sampling.  These results provide a theoretical foundation for the use of the Griddy Gibbs sampling method. 

Moreover, using the theoretical framework developed to validate the Griddy Gibbs sampling method, we also successfully provided a more general result about the sensitivity of invariant measures under small perturbations on the transition probability. Our results imply that if we replace the transitional probability $P$ of a Monte Carlo Markov chain by a different transitional probability $Q$ that is close to $P$ in $L^p$ norm ($2 \le p \le \infty$), the distance between the two invariant measures (in $L^p$) is bounded by a constant times the $L^p$-distance between $P$ and $Q$,  provided that the approximation schemes satisfy a mild condition provided in the paper. This condition is very general and does not hinder the application of the Griddy Gibbs sampling method, since it can always be guaranteed simply by using additional cutoff functions on the approximation scheme, without significantly affecting its accuracy. The method can be generalized to validate other Monte Carlo Markov chain sampling methods that involve approximation. 

We also gave numerical examples to illustrate our theoretical findings and demonstrate the utility of the method in different applications. The numerical results confirm the linear relation between the distance between the invariant measures and the accuracy of the approximation scheme derived in theory. Moreover, our examples illustrate that Griddy Gibbs performs as well as its variants in applications and that the algorithm is  simpler to implement and less computationally expensive. Additionally, the Markov chains generated by this algorithm have significantly smaller auto-correlation coefficients than those of other variant algorithms. These features demonstrate that Griddy Gibbs is a simple and effective sampling method that can be employed in applications with confidence in its validity.

\section*{Acknowledgement}
All authors of this paper are partially supported by NSF grant DMS-0900277.

\end{document}